\documentclass[11pt,a4paper]{amsart}
\usepackage{amssymb,amscd,amsmath,young, stmaryrd}
\usepackage{mathrsfs, mathdots} \usepackage[mathcal]{eucal}
\usepackage{float,mathabx} \usepackage[usenames]{color} \usepackage{soul}
\usepackage{hyperref}

  \theoremstyle{plain}
  \newtheorem{thm}{Theorem}[section]
  \newtheorem{lem}[thm]{Lemma}
  \newtheorem{pro}[thm]{Proposition}
  \newtheorem{cor}[thm]{Corollary}
  
  \newtheorem{thmABC}{Theorem}

  \newtheorem*{con*}{Conjecture}
  \newtheorem{lemma}[thm]{Lemma}
  \newtheorem{con}[thmABC]{Conjecture}

  \theoremstyle{remark}
  \newtheorem{rem}[thm]{Remark}
  
  \newtheorem{exm}[thm]{Example}

  \newtheorem*{acknowledgements}{Acknowledgements}

  \numberwithin{equation}{section}
  \numberwithin{table}{section}

  \newcommand{\N}{\mathbb{N}}
  \newcommand{\Z}{\mathbb{Z}}
  \newcommand{\Q}{\mathbb{Q}}
  
  \newcommand{\C}{\mathbb{C}}
  
  \newcommand{\R}{\mathbb{R}}

  \renewcommand{\epsilon}{\varepsilon}
  
  \renewcommand{\phi}{\varphi}
  \renewcommand{\theta}{\vartheta}

  \newcommand{\rarr}{\rightarrow}

  \DeclareMathOperator{\sd}{sd}
  \DeclareMathOperator{\Hilb}{Hilb}
  \DeclareMathOperator{\red}{red}
  \DeclareMathOperator{\maj}{maj}
  \DeclareMathOperator{\des}{des}

  \DeclareMathOperator{\Des}{Des}

  \DeclareMathOperator{\real}{Re}

  \def \dd {\mathsf{d}}
  \def \FF {\mathsf{F}}
  \def \OO {\mathsf{O}}
  \def \pp {\mathsf{p}}

  \def \bfm {{\bf m}}

  \def \bfone {{\bf 1}}
  \def \bftwo {{\bf 2}}
  \def \bfthree {{\bf 3}}
  \def \bfm {{\bf m}}
  \def \wt {\widetilde}

  \def \Zp  {\mathbb{Z}_p}

  \author{Angela Carnevale} \address{Fakult\"at f\"ur Mathematik,
    Universit\"at Bielefeld, D-33501 Bielefeld, Germany}
  \email{acarneva1@math.uni-bielefeld.de, C.Voll.98@cantab.net}
  \author{Christopher Voll} 

  \keywords{Orbit Dirichlet series, multiset permutations, Carlitz's
    $q$-Eulerian polynomials, Hadamard products of rational generating
    functions, Igusa functions, local functional equations, natural
    boundaries}

\subjclass[2000]{37C30, 37P35, 30B50, 11M41, 05A15, 05A19}

\begin{document}
   \title{Orbit Dirichlet series and multiset permutations} \date{\today}

\begin{abstract} 
  We study Dirichlet series enumerating orbits of Cartesian products
  of maps whose orbit distributions are modelled on the distributions
  of finite index subgroups of free abelian groups of finite rank. We
  interpret Euler factors of such orbit Dirichlet series in terms of
  generating polynomials for statistics on multiset permutations,
  viz.\ descent and major index, generalizing Carlitz's $q$-Eulerian
  polynomials.

  We give two main applications of this combinatorial
  interpretation. Firstly, we establish local functional equations for
  the Euler factors of the orbit Dirichlet series under
  consideration. Secondly, we determine these (global) Dirichlet
  series' abscissae of convergence and establish some meromorphic
  continuation beyond these abscissae. As a corollary, we describe the
  asymptotics of the relevant orbit growth sequences. For Cartesian
  products of more than two maps we establish a natural boundary for
  meromorphic continuation. For products of two maps, we prove the
  existence of such a natural boundary subject to a combinatorial
  conjecture.
\end{abstract}
  \maketitle

\thispagestyle{empty}
\section{Introduction and main results}
Let $X$ be a space and $T:X\rarr X$ a map. A \emph{closed orbit of
  length $n\in\N$} is a set of the form
$$\{x,T(x),T^2(x),\dots,T^n(x)=x\}$$
of cardinality $n$. Assume that the number $\OO_T(n)$ of closed orbits
of length $n$ under $T$ is finite for all $n\in\N$. The \emph{orbit
  Dirichlet series} of $T$ is the Dirichlet generating
series $$\dd_T(s) = \sum_{n=1}^\infty \OO_T(n)n^{-s},$$ where $s$ is a
complex variable.

If $T$ has a single closed orbit of each length $n$, then $\dd_T(s)$
is just Riemann's zeta function $\zeta(s) = \sum_{n=1}^\infty n^{-s}$.
If, more generally, $T=T_r$ is such that the number of closed orbits
of length $n$ equals the number $a_n(\Z ^r)$ of subgroups of $\Z^r$ of
index $n$ for all $n\in\N$, then $\dd_{T_r}(s)$ is the well known
Dirichlet generating series (or ``zeta function'') $\zeta_{\Z^r}(s)$
enumerating subgroups of finite index of the free abelian group $\Z^r$
of rank~$r$. More precisely,
  \begin{equation}\label{zetafree}
    \dd_{T_r}(s) = \zeta_{\Z^r}(s)= \sum_{n=1}^\infty a_n(\Z^r)n^{-s}
    =\prod_{i=0}^{r-1}\zeta(s-i);
  \end{equation}
  cf.\ \cite[Proposition~1.1]{GSS/88}.  

Let $\lambda= (\lambda_1,\dots,\lambda_m)\in \N ^m$ with $\lambda_1
\geq \dots \geq \lambda_m\geq 1$ be a partition of
$N=\sum_{i=1}^m\lambda_i$. For $i=1,\dots,m$, let $T_{\lambda_i}$
be a map as above with $\dd_{T_{\lambda_i}}(s) =
\zeta_{\Z^{\lambda_i}}(s)$. We write
$${T}_{\lambda} ={ T}_{\lambda_1} \times \dots \times { T}_{\lambda_m}$$
for the Cartesian product of the maps $T_{\lambda_i}$. Clearly, the
arithmetic function $n \mapsto \OO_{ T_\lambda}(n)$ is
multiplicative, whence
  \begin{equation*}\label{equ:euler}
  \dd_{T_\lambda}(s) = \prod_{p \textup{ prime}} \dd_{ T_\lambda,p}(s),
  \end{equation*}
  where, for a prime~$p$,
  $$\dd_{ T_\lambda,p}(s) = \sum_{k=0}^\infty \OO_{T_\lambda}(p^k)p^{-ks}.$$

  We remark that maps $ T_{\lambda_i}$ as above exist, {even if
    they are required to be smooth. Indeed,} by a result of Windsor,
  any sequence $(a_n)_{n\geq 1}$ of nonnegative integers may be
  realized as the sequence $(\OO_T(n))_{n\geq 1}$ for a suitable
  $C^{\infty}$-diffeomorphism $T$ of the $2$-dimensional torus
  $X=\mathbb{T}^2 = (\mathbb{R}/\mathbb{Z})^2$; cf.~\cite{Windsor/08}.
  
  In this paper we prove and exploit combinatorial formulae for the
  Euler factors of orbit Dirichlet series of the form
  $\dd_{T_\lambda}(s)$ above using generating polynomials for
  statistics on multiset permutations.

Our first main result is phrased in terms of the bivariate polynomial
$C_\lambda\in\Z[x,q]$ giving the joint distribution of the statistics
$\des$ and $\maj$ on~$S_{\lambda}$, the set of multiset permutations
of the multiset
$\{\underbrace{\bfone,\dots\bfone}_{\lambda_1},\underbrace{\bftwo,\dots\bftwo}_{\lambda_2}\ldots,\underbrace{\bfm,\dots,\bfm}_{\lambda_m}\}$. See
Section~\ref{sec:multip} for precise definitions.

\begin{thm}\label{thm:partition}
 Let $\lambda=(\lambda_1,\ldots,\lambda_m)$ be a partition
 of~$N$. Then
 \begin{equation}\label{equ:partition} 
  \dd_{T_{\lambda}}(s) = \prod_{p \textup{ prime}}
  \frac{C_{\lambda}(p^{-1-s},p)}{\prod_{i=1}^{N}(1-p^{i-1-s})}=
  \prod_{p \textup{ prime}}\frac{\sum_{w \in
      S_{\lambda}}p^{(-1-s)\des(w)+\maj(w)}}{\prod_{i=1}^{N}(1-p^{i-1-s})}.
 \end{equation}
\end{thm}

Key to Theorem~\ref{thm:partition} is an identity, essentially due to
MacMahon, for Hadamard products of certain rational generating
functions. {It is well known that if $A(x)=\sum_{k=0}^\infty
  a_kx^k$ and $B(x) = \sum_{k=0}^\infty b_kx^k\in\Q(x)$ are rational
  functions, then their Hadamard product $(A\ast B)(x) =
  \sum_{k=0}^\infty a_kb_k x^k$ is also a rational function;
  cf.\ \cite[Proposition~4.2.5]{Stanley/12}. Given rational functions
  $A_1(x),\dots,A_m(x)$, we write $\Asterisk_{i=1}^m A_i(x)$ for their
  Hadamard product $A_1(x)\ast \dots \ast
  A_m(x)$.}

\begin{pro}[MacMahon; cf.\ Remark~\ref{rem:macmahon}]\label{pro:macmahon}
Let $\lambda=(\lambda_1,\ldots,\lambda_m)$ be a partition of~$N$. Then
$$\Asterisk_{i=1}^m \prod_{k=0}^{\lambda_i} \frac{1}{1-q^{k}x} =
  \frac{C_{\lambda}(x,q)}{\prod_{i=0}^N(1-xq^{i})} \in \Q(x,q).$$
\end{pro}

We call a partition of the form $\lambda=(r,\ldots,r)=(r^m)$ a
\emph{rectangle}. We
use Theorem~\ref{thm:partition} to prove that the Euler factors
in~\eqref{equ:partition} satisfy {certain} functional equations
upon inversion of the prime if and only if the partition $\lambda$ is
a rectangle. {We denote by $T_r^{\times m}$ the $m$-fold Cartesian
  power $T_r \times \dots \times T_r$.}
\begin{thm}\label{cor:funeq} 
Let $p$ be a prime. For all $r,m\in\N$,
\begin{equation}\label{equ:fun}
\dd_{T_r^{\times m},p}(s)\vert_{p \rarr p^{-1}} = (-1)^{rm}
p^{m\binom{r+1}{2}-r-rs} \dd_{{ T}_r^{\times m},p}(s).
\end{equation}
If $\lambda$ is not a rectangle, then $\dd_{ T_\lambda,p}(s)$ does
not satisfy a functional equation of the form
\begin{equation}\label{equ:funeq.dir}
\dd_{ T_\lambda,p}(s)\vert_{p \rarr p^{-1}} = \pm p^{d_1-d_2s}
\dd_{ T_{\lambda},p}(s)
\end{equation} for $d_1,d_2\in\N_0$.
\end{thm}

We prove Theorems~\ref{thm:partition} and \ref{cor:funeq} in
Section~\ref{sec:proof}. The functional equations \eqref{equ:fun} are
deduced from the combinatorial properties of the polynomials
$C_\lambda$ studied in Section~\ref{sec:multip}.

In Section~\ref{sec:analytic} we collect a number of corollaries about
the analytic properties of the orbit Dirichlet series that we
study. In particular, we determine the abscissa of convergence of
$\dd_{T_{\lambda}}(s)$ and establish meromorphic continuation beyond
this abscissa. A standard application of a Tauberian theorem then
yields an asymptotic result on the growth of the (partial sums of the)
numbers $\OO_{T_{\lambda}}(n)$; see Theorem~\ref{analytic}.

If $\lambda=(r)$ or $\lambda=(1,1)$, then $\dd_{T_{\lambda}}(s)$ has
meromorphic continuation to the whole complex plane. In contrast, for
partitions with more than two parts{---}pertaining to Cartesian products of more than
two maps{---}or two parts of equal length greater than $1$ we
establish a natural boundary for meromorphic continuation
at~$\left(\sum_{i=1}^m\lambda_i\right)-2$. For partitions with two
parts of unequal lengths, we establish such a natural boundary subject
to a combinatorial conjecture on some special values of the
polynomials $C_{\lambda}$ discussed in Section~\ref{subsec:gen.poly};
see Theorem~\ref{thm:nat.bound}.

In Section~\ref{sec:one} we concentrate on partitions of the form
$\lambda=(1^m)$, pertaining to the $m$-th Cartesian power of a map
with orbit Dirichlet series~$\dd_{T_1}(s)=\zeta(s)$.  Orbit
Dirichlet series of products of such maps were previously studied, for
very special cases, in \cite{PakapongpunWard/14}. The result
\cite[Theorem~4.1]{PakapongpunWard/14}, for instance, is the special
case $\lambda=(1,1,1)$ of our Theorem~\ref{thm:nat.bound}; see also
Section~\ref{sec:one}. For partitions of the form $\lambda=(1^m)$ the
polynomial $C_\lambda$ is the well-studied Carlitz's $q$-Eulerian polynomial,
enumerating the elements of the symmetric group by the statistics
$\des$ and~$\maj$.  We also observe that in this case the Euler
factors of \eqref{equ:partition} are {Igusa functions} in the
terminology of \cite{SV1/15}.

In Section~\ref{sec:red} we consider ``reduced'' orbit Dirichlet
series and note some connections with the theory of $h$-vectors of
simplicial complexes.

Dirichlet generating series are widely used in enumerative problems
arising in algebra, geo\-metry, and number theory. Orbit Dirichlet
series as defined above are studied for instance in~\cite{EMST/10}. 
Local functional equations such as the ones established in
Theorem~\ref{cor:funeq} occur frequently in the theory of zeta
functions of {groups, rings, and modules; see, for example,
  \cite{Voll/10, Voll/17}}. In the cases where they are explained
combinatorially, they may often be traced back to functional equations
satisfied by Igusa-type functions; see, for instance,
\cite{KlopschVoll/09, SV1/15}.

\subsection{Notation}
We write $\N=\{1,2,\dots\}$ and, for a subset $I\subseteq \N$, set
$I_0 = I \cup \{0\}$. Given $n\in\N$, we write $[n]=\{1,\dots,n\}$ and
$n-I= \{n-i\mid i\in I\}$.  For $I=\{i_1,\ldots,i_r\}\subseteq [n-1]$
with $i_1 \leq \dots \leq i_r$ we let
$$\binom{n}{I}=\frac{n!}{i_1! (i_2 - i_1)!\cdots(n-i_r)!} $$
denote the multinomial coefficient. Given $a \geq b\in\N_0$ and a
variable $q$, we write
  $$\binom{a}{b}_q = \prod_{i=1}^b \frac{q^{a-b+i}-1}{q^i-1}\in\Z[q]$$
  for the $q$-binomial coefficient.

  \section{Permutations of multisets}\label{sec:multip}
  In this section we set up notation and prove some basic facts
  regarding multiset permutations (see also \cite[Section
  5.1.2]{Knuth/98}).

  \subsection{Multiset permutations} \label{subsec:multiset.permut}
  Let $\lambda = (\lambda_1,\dots,\lambda_m)$ be a partition of $N
  =\sum_{i=1}^m\lambda_i$.  The \emph{multiset} $$A_{\lambda}=
  \{\underbrace{\bfone,\dots\bfone}_{\lambda_1},\underbrace{\bftwo,\dots\bftwo}_{\lambda_2}\ldots,\underbrace{\bfm,\dots,\bfm}_{\lambda_m}\}$$
  comprises $\lambda_1$ (indistinguishable) copies of the ``letter''
  $\bfone$, $\lambda_2$ copies of the ``letter'' $\bftwo$ etc. A
  \emph{multiset permutation} (or \emph{multipermutation}) on
  $A_{\lambda}$ is a word $w= w_1 \dots w_N$ formed with all the $N$
  elements of $A_{\lambda}$. We denote {by} $S_{\lambda}$ the set
  of all multiset permutations on~$A_\lambda$.  If
  $\lambda=(1,\dots,1)=(1^m)$, then we recover the set $S_{m}$ of
  permutations of the set $A_{(1^m)} = \{\bfone,\bftwo,\dots,\bfm\}$.

  In general, $S_{\lambda}$ lacks a natural group structure, but a
  number of classical statistics on the Coxeter group $S_m$ have
  analogues for general partitions.  For instance, one defines the
  \emph{descent set} $\Des(w)$ of $w = \prod_{i=1}^N w_i\in S_\lambda$
  as
  $$\Des(w)=\{i \in [N-1] \mid w_i>w_{i+1}\},$$ where, of course, one
  uses the ``natural'' ordering $\bfm > \dots > \bftwo > \bfone$ on
  the letters of $A$.  The \emph{descent} and \emph{major index}
  statistics on $S_\lambda$ are defined via 
$$\des(w)=|\Des(w)| \quad\textup{ and }\quad\maj(w)=\sum_{i\in\Des(w)} i.$$ The
``trivial word'' $\bfone^{\lambda_1}\bftwo^{\lambda_2}\ldots
\bfm^{\lambda_m}$ is clearly the unique element in $S_{\lambda}$ with
empty descent set.

  \begin{exm}
    For $\lambda= (3,3,1)$, the element $w=\bfone \bftwo
    \bfone\bftwo\bfthree \bfone\bftwo\in S_{\lambda}$ has $\Des(w) =
    \{2,5\}$, whence $\des(w)=2$ and $\maj(w)=7$.
  \end{exm}

\begin{rem}
  One may, more generally, consider multisets indexed by compositions,
  rather than partitions, of~$N$. As we are interested in the joint
  distribution of $\des$ and $\maj$, the order of the parts does not
  matter to us (cf.~\eqref{eq:mmpart} below), so we only consider
  partitions.
\end{rem}

Recall that we call $\lambda$ a rectangle if $\lambda_1 = \dots =
\lambda_m=r$, say, viz.\ $\lambda = (r^m)$. In this case, we write
$S_{r,m}$ for $S_{(r^m)}$. If, moreover, $r=1$, then we write $S_m$
for~$S_{1,m} = S_{(1^m)}$, the (set underlying the) symmetric group of
degree $m$.

\begin{lem} \label{lem:unique.longest} The partition $\lambda$ is a
  rectangle if and only if there exists a unique element of
  $S_\lambda$ at which $\des$ attains its maximum. If $\lambda=(r^m)$,
  then both $\des$ and $\maj$ take their maximal values at $w_0 =
  \left(\bfm \dots \bftwo \bfone\right)^r$ of $S_{\lambda}$, viz.\
  $\des(w_0)=r(m-1)$ and $\maj(w_0)=r^2\binom{m}{2}$.
\end{lem}

\begin{proof}
 Set $s=\lambda_1$ and write $\mu=(\mu_1,\dots,\mu_s)$ for the dual
 partition of $\lambda$. Thus $m = \mu_1 \geq \dots \geq \mu_s \geq
 1$. The statistic $\des(w)$ attains its maximal value $\sum_{\sigma
   =1}^s (\mu_{\sigma}-1)$ precisely at the word
 $$w_0= (\boldsymbol{\mu_1}\dots \bftwo
 \bfone)(\boldsymbol{\mu_2}\dots \bftwo \bfone) \dots
 (\boldsymbol{\mu_{s}}\dots \bftwo \bfone)$$ and all the elements of
 $S_\lambda$ obtained from $w_0$ by permuting the $s$ ``blocks''
 $\boldsymbol{\mu_{\sigma}}\dots \bftwo \bfone$, $\sigma\in[s]$. All
 these elements coincide if and only if $\lambda$ is a rectangle, say
 $\lambda=(r^m)$. In this case, $\mu = (m^r)$ and $w_0$ satisfies
 $\des(w_0)= r(m-1)$ and $\maj(w_0)=\binom{rm}{2} - m\binom{r}{2} =
 r^2\binom{m}{2}$.
\end{proof}

We define the involution
\begin{align}\label{def:map.circ}
   \phantom{x}^\circ: S_{r,m} \rarr S_{r,m},\quad w = w_1\dots w_N
   \mapsto w^\circ = ({\bf m+1} - w_N) \dots ({\bf
     m+1}-w_1)\end{align} which ``reverses and inverts'' the elements
of $S_{r,m}$. 

\begin{rem}
  If $r=1$, then $w_0\in S_m$ is the ``longest element'' (with respect
  to Coxeter length) and $^\circ$ is just conjugation by~$w_0$.
\end{rem}
We collect some properties of this involution in the following
elementary and easy lemma, whose proof we omit.
\begin{lem}\label{lem:inv}
For all
  $w\in S_{r,m}$ the following hold.
  \begin{enumerate}
  \item $\Des(w^\circ) = rm - \Des(w)$,
   \item $\des(w^\circ) = \des(w)$,
   \item $\maj(w^\circ) = \des(w)rm - \maj(w)$.
  \end{enumerate}\end{lem}

  \subsection{Generating polynomials}\label{subsec:gen.poly}
  Let $x$ and $q$ be variables and set
  \begin{equation}\label{equ:C.general.part}
    C_{\lambda}(x,q)=\sum_{w \in
      S_{\lambda}}x^{\des(w)}q^{\maj(w)}\in \Z [x,q].
  \end{equation}
  A result of MacMahon (\cite[\textsection 462, Vol.~2, Ch.~IV,
  Sect.~IX]{MacMahon/16}) states that, in $\Q(x,q) \cap
  \Q(q)\llbracket x \rrbracket$,
  \begin{equation}\label{eq:mmpart}
    \sum_{k=0}^\infty\left(\prod_{i=1}^m\binom{\lambda_i +k}{k}_q\right) x^k=\frac{C_{\lambda}(x,q)}{\prod_{i=0}^{N}(1-xq^i)}.
  \end{equation}

  If $\lambda=(r^m)$ is a rectangle, then we write $C_{r,m}$ for
  $C_{(r^m)}$. If, moreover, $r=1$, then we write $C_m$ for
  $C_{1,m}=C_{(1^m)}$. In this case, \eqref{equ:C.general.part}
  defines Carlitz's $q$-Eulerian polynomial
  (\cite{Carlitz/54,Carlitz/75})
  $$C_{m}(x,q)=\sum_{w \in S_m}x^{\des(w)}q^{\maj(w)}\in \Z [x,q].$$
Note that
\begin{equation}\label{equ:eulerian}
C_m(x,1) = \sum_{w\in S_m}x^{\des(w)} = A_m(x)/x\in\Z[x],
\end{equation} where
  $A_m(x)$ is the $m$-th Eulerian polynomial;
cf.\ \cite[Section~1.4]{Stanley/12}.

  \begin{exm}
    For $\lambda=(2,1)$, $S_\lambda=\{\bfone \bfone \bftwo, \bfone
    \bftwo \bfone, \bftwo \bfone \bfone \}$, so
  $$C_{(2,1)}(x,q)=1+xq+xq^2.$$
  For $r=m=2$,
    $S_{2,2} = \{\bfone \bfone \bftwo \bftwo,\,\bfone \bftwo \bftwo
    \bfone,\,\bfone \bftwo \bfone \bftwo,\,\bftwo \bfone \bfone
    \bftwo,\, \bftwo \bftwo \bfone \bfone ,\,\bftwo \bfone \bftwo
    \bfone\}$, whence
  $$C_{2,2}(x,q) = 1 + x q + 2x q^2 + x q^3 + x^2 q^4.$$ 
  Finally, for $m=3$ resp.\ $m=4$,
  \begin{align}
   C_3(x,q)&=1+2xq+2xq^2+x^2q^3, \textup{ resp.\ } \nonumber\\
    C_4(x,q)&=(1+xq^2)(1+3xq+4xq^2  +3xq^3+x^2q^4).\nonumber
  \end{align}
  \end{exm}
  
  To establish some of the analytic properties of $\dd_{
    T_{\lambda}}(s)$ in Section~\ref{sec:analytic}, we need a
  description of the unitary factors of the bivariate polynomials
  $C_{\lambda}(x,q)$. Here, a polynomial $f\in\Z[x,q]$ is called
  \emph{unitary} if it is {nonconstant}
  and there exists $F\in\Z[Y]$ such that all complex roots of $F$ have
  absolute value $1$ and $f(x,q) = F(x^aq^b)$ for some $a,b\in\N_0$.

  As $\maj(w) > 0 $ implies $\des(w) > 0$ for all $w\in S_\lambda$,
  unitary factors of $C_\lambda (x,q)\in\Z[x,q]$ give rise to unitary
  factors of
  \begin{equation}\label{equ:C.descent}
    C_\lambda (x,1) = \sum_{w\in S_{\lambda}} x^{\des(w)}\in\Z[x].
    \end{equation}

The following Lemma describes the occurrence of unitary factors of
Carlitz's $q$-Eulerian polynomials, pertaining to partitions of the form
$\lambda = (1^m)$.

\begin{lem}\label{lem:factor}
  Carlitz's $q$-Eulerian polynomial $C_m(x,q)\in\Z[x,q]$ has a unitary
  factor if and only if $m$ is even. If $m=2k$, then
 \begin{equation*}\label{evenfact}
    C_m(x,q)=(1+x q^k) C_m ' (x,q),
 \end{equation*}
 where $C_m'(x,q)=\sum_{w\in S_m^{\{k\}}} x^{\des(w)}q^{\maj(w)}$ and
 $S_m^{\{k\}}$ is the parabolic quotient $S_m^{\{k\}}=\{w \in S_m \mid
 \Des(w)\subseteq [m-1]\setminus \{k\}\}$. Moreover, $C'_m(x,q)$ has
 no unitary factor.
\end{lem}

\begin{proof}
  By a result of Frobenius (\cite[p.~829]{Frobenius/1910}), the roots
  of $C_m(x,1)$ are all real, simple, and negative; moreover, $-1$ is
  a root only for even~$m$. Thus the $q$-Eulerian polynomials
  $C_m(x,q)$ have unitary factors only for even $m$. Let $m=2k$ and
  denote by $w_0$ the longest element of~$S_m$. The map $S_m^{\{k\}}
  \rarr S_m \setminus S_m^{\{ k \}}, w \mapsto w w_0$, is obviously a
  bijection. Hence
 \begin{align*}
    C_m(x,q)&=\sum_{w\in S_m}x^{\des(w)}q^{\maj(w)}=\sum_{w\in S_m}
    \prod_{j\in\Des(w)} xq^j\\ &=\sum_{w\in S_m^{\{k\}}}
    \prod_{j\in\Des(w)} xq^j+ \sum_{w\in S_m \setminus S_m^{\{k\}}}
    \prod_{j\in\Des(w)} xq^j \\ &=\sum_{w\in S_m^{\{k\}}}
    \prod_{j\in\Des(w)} xq^j+ x q^k \sum_{w\in S_m \setminus
      S_m^{\{k\}}} \prod_{\substack{ j\in\Des(w)\\ j \neq k}}
    xq^j\\ &=\sum_{w\in S_m^{\{k\}}} \prod_{j\in\Des(w)} xq^j+ x q^k
    \sum_{w\in S_m^{\{k\}}} \prod_{ j\in\Des(w)} xq^j\\ &=(1+x
    q^k)\sum_{w\in S_m^{\{k\}}} \prod_{j\in\Des(w)} xq^j = (1+x q^k)
    C_m'(x,q).
 \end{align*}
 The non-existence of unitary factors of $C_m'(x,q)$ follows again
 from the simplicity of $x=-1$ as a root of $C_m(x,1)$.
\end{proof}

\begin{rem}
 The polynomials $C_m(x,1)$, for $m$ odd, resp.\ $C_m(x,1)/(1+x)$, for
 $m$ even, have been conjectured to be irreducible for all $m$; for a
 discussion and proofs of irreducibility in various special cases,
 see~\cite{Heidrich/84}.
\end{rem}

\begin{rem}Consider again a general partition $\lambda$. Generalizing the result
of Frobenius referred to in the proof of Lemma~\ref{lem:factor}, all
zeros of the polynomials $C_\lambda(x,1)$ are real, simple, and
negative; see \cite[Corollary~2]{Simion/84}. By the above discussion,
a necessary condition for the occurrence of unitary factors of
$C_{\lambda}(x,q)$ is hence that $C_\lambda(-1,1)=0$. We remark that
in the case that $\lambda=(r^m)$ is a rectangle, $C_\lambda (-1,1)$
is, up to a sign, the so-called Charney-Davis quantity of the graded
poset of the disjoint union of $m$ labelled chains of length $r$;
see~\cite{RSW/03}.
\end{rem}
For our applications in Section~\ref{sec:analytic} we require
statements about the (non-)existence of unitary factors of
$C_{\lambda}(x,q)$ principally in the case $m=2$, on which we focus
for most of the reminder of this section. Recall that
$C_{(\lambda_1,\lambda_2)}(x,1)$ is the descent polynomial of
$S_{(\lambda_1,\lambda_2)}$; cf.~\eqref{equ:C.descent}. In
\cite[\textsection 144-146, Vol.~1, Ch.~II, Sect.~IV]{MacMahon/16}
MacMahon gives three proofs of the following lemma.

\begin{lemma}[MacMahon]\label{lem:lambda} 
 Let $\lambda=(\lambda_1,\lambda_2)$. Then
$$C_{(\lambda_1,\lambda_2)}(x,1) = \sum_{j=0}^{\lambda_2}  \binom{\lambda_1}{j}\binom{\lambda_2}{j}x^j.$$
\end{lemma}

In terms of Jacobi polynomials,
$$C_{(\lambda_1,\lambda_2)}(x,1)=(1-x)^{\lambda_2}
P_{\lambda_2}^{(0,\lambda_1-\lambda_2)}\left(\frac{1+x}{1-x}\right);$$ cf.\ \cite[eq.~(1.2.7)]{GasperRahman/04}.
It follows from MacMahon's third proof of Lemma~\ref{lem:lambda}
(cf.~\cite[\textsection 146]{MacMahon/16}) that the number of elements
in $S_{(\lambda_1,\lambda_2)}$ with $k$ descents equals the number of
elements with $k$ occurrences of $\bftwo$ in the first $\lambda_1$
positions. We conjecture the following.
\begin{con}\label{con:nonzero}
 Let $\lambda_1>\lambda_2$. Then the following equivalent statements hold:
\begin{enumerate}
\item $C_{(\lambda_1,\lambda_2)}(-1,1)\neq 0$,
\item 
$P_{\lambda_2}^{(0,\lambda_1-\lambda_2)}(0)\neq 0$,
\item \begin{align*}&|\{w\in S_{(\lambda_1,\lambda_2)}\mbox{ with an
    even number of $ \bftwo$s in the first $\lambda_1$ positions}
  \}|\neq\\ &|\{w\in S_{(\lambda_1,\lambda_2)}\mbox{ with an odd
    number of $ \bftwo$s in the first $\lambda_1$
    positions}\}|.\end{align*}\end{enumerate} In particular,
$C_{(\lambda_1,\lambda_2)}(x,q)$ has no unitary factor.
\end{con}

D.\ Stanton pointed out to us that
$C_{(\lambda_1,\lambda_2)}(-1,1) \neq 0$ for all $\lambda_2$ and
$\lambda_1 > \lambda_2( \lambda_2+1)-1$, since the alternating
summands have increasing absolute values; L.~Habsieger informed us of
unpublished work of his establishing Conjecture~\ref{con:nonzero} for
all $\lambda_1 \geq c \lambda_2$ for some constant $c>1$. The quantity
$C_{(\lambda_1,\lambda_2)}(-1,1)$ may also be expressed in terms of
Krawtchouk polynomials; it is equal to
$(-1)^{\lambda_2}k_{\lambda_2}(\lambda_2,2,\lambda_1+\lambda_2)$ in
the notation of~\cite{ChiharaStanton/90}.

If $\lambda_1=\lambda_2$, then $C_{(\lambda_1,\lambda_2)}(-1,1) \neq
0$ if and only if the $\lambda_i$ are even (see
\cite[Proposition~2.4]{RSW/03}), whence
$C_{(\lambda_1,\lambda_2)}(x,q)$ has no unitary factors in this
case. In the odd case, the following holds.

\begin{pro}[{\cite[Proposition 5]{Carnevale/16}}]\label{pro:shape}
 Let $\lambda_1=\lambda_2=r\equiv 1 \pmod 2$. Then
$$C_{r,2}(x,q)=(1+xq^{r})C'_{r,2}(x,q),$$ where $C'_{r,2}(x,q)$ has no
 unitary factor.
\end{pro}
 Generalizing Lemma~\ref{lem:factor}, Conjecture~\ref{con:nonzero} and
Proposition \ref{pro:shape}, we put forward the following

 \begin{con}\label{con:master}
   Let $\lambda = (\lambda_1,\dots,\lambda_m)$ be a partition. Then
   $C_{\lambda}(x,q)$ has a unitary factor if and only if $\lambda =
   (r^m)$ is a rectangle, with $r$ odd and $m$ even. In this
   case, $$C_{r,m}(x,q) = (1+xq^{\frac{rm}{2}})C'_{r,m}(x,q)$$ and
   $C'_{r,m}(x,q)$ has no unitary factor.
 \end{con}

 \subsection{Functional equations}

 \begin{pro}\label{pro:funeq.C.rectangle}
   For all $r,m\in\N$,
 \begin{equation*}\label{rmbiv}
   C_{r,m}(x^{-1},q^{-1})=    x^{-r(m-1)}q^{-r^2 \binom{m}{2}} C_{r,m}(x,q).
 \end{equation*}
 If $\lambda$ is not a rectangle, then $C_{\lambda}(x,q)$ does not
 satisfy a functional equation of the form
 \begin{equation}\label{equ:funeq.gen}
  C_{\lambda}(x^{-1},q^{-1}) = x^{-d_1}q^{-d_2} C_{\lambda}(x,q),
\end{equation} for
$d_1,d_2\in\N_0$.
 \end{pro}

\begin{proof}
  As $C_{\lambda}(x,1)\in\Z[x]$ has constant term $1$, a necessary
  condition for $C_\lambda$ to satisfy a functional equation of the
  form \eqref{equ:funeq.gen} is that $C_\lambda(x,1)$ is monic. By
  Lemma~\ref{lem:unique.longest}, this holds if and only if $\lambda$
  is a rectangle. This proves the second statement.

  To establish the first statement, let $r,m\in\N$. For
  $i\in[r(m-1)]_0$, we set $$C^{(i)}_{r,m}(q)=\sum_{\{w \in S_{r,m}
    \mid \des(w)=i\}}q^{\maj(w)}\in\Z[q],$$ so that $C_{r,m}(x,q) =
  \sum_{i=0}^{r(m-1)}C_{r,m}^{(i)}(q)x^i.$ With the map
  $\phantom{}^\circ$ defined in \eqref{def:map.circ},
  Lemma~\ref{lem:inv} yields
  \begin{equation}\label{majsym}\nonumber
    C^{(i)}_{r,m}(q^{-1})=
    q^{-irm}\sum_{\substack{\{ w\in
        S_{r,m}\mid \\
        \des(w)=i\}}}q^{irm-\maj(w)}=q^{-irm}\sum_{\substack{\{ w\in
        S_{r,m}\mid \\ \des(w)=i\}}}q^{\maj(w^\circ)} = q^{-irm}
    C^{(i)}_{r,m}(q).
  \end{equation}
  Using this and the relations
  \begin{equation*}\label{mm}
    C^{(r(m-1)-i)}_{r,m}(q)=q^{r^2\binom{m}{2}-irm} C^{(i)}_{r,m}(q)
  \end{equation*}
  (cf. \cite[\textsection 461, Vol.~2, Ch.~IV, Sect.~IX]{MacMahon/16})
  we obtain
 \begin{align*}
   C_{r,m}(x^{-1},q^{-1})&=\sum_{i=0}^{r(m-1)}
   C_{r,m}^{(i)}(q^{-1})x^{-i}=\sum_{i=0}^{r(m-1)}
   q^{-irm}C_{r,m}^{(i)}(q)x^{-i}\\&=\sum_{i=0}^{r(m-1)}
   q^{-irm}q^{-r^2\binom{m}{2}+irm}C_{r,m}^{(r(m-1)-i)}(q)x^{-i}\\&=
   q^{-r^2\binom{m}{2}}\sum_{j=0}^{r(m-1)} C_{r,m}^{(j)}(q)
   x^{-r(m-1)+j}\\ &=
   x^{-r(m-1)}q^{-r^2\binom{m}{2}}C_{r,m}(x,q).\qedhere
 \end{align*}
 \end{proof}

 \section{Proofs of Theorems~\ref{thm:partition} and \ref{cor:funeq}}\label{sec:proof}
\subsection{Proof of Theorem~\ref{thm:partition}}

Let $r\in\N$. 
Recall that
$ \zeta_{\Z^r}(s) = \prod_{p \textup{ prime}}\zeta_{\Zp^r}(s)$, where,
for a prime $p$, the Euler factor
$\zeta_{\Zp^r}(s) = \sum_{k=0}^\infty a_{p^k}(\Zp^r)p^{-ks}$
enumerates the $\Zp$-submodules of finite additive index in
$\Zp^r$. Here, $\Zp$ denotes the ring of $p$-adic integers.
 \begin{lemma}[cf., e.g., \cite{Gruber/97}]\label{lem:subgroups}
   For all $k\in\N_0$,
 \begin{equation*}\label{equ:product}
   a_{p^k}(\Zp ^r) = \binom{r-1+k}{k}_p.
 \end{equation*}
 \end{lemma}

\begin{proof}
  For a variable $t$,
 \begin{equation}\label{equ:subgroups.abel}\sum_{k=0}^\infty a_{p^k}(\Zp ^r)t^k =
  \frac{1}{\prod_{i=1}^r(1-p^{i-1}t)}= \sum_{k=0}^\infty
  \binom{r-1+k}{k}_pt^k;
  \end{equation}see \eqref{zetafree} for the first equality
  and, for instance, \cite[Ch.\ 1.8]{Stanley/12} for the second.
\end{proof}

\begin{rem}\label{rem:macmahon}
Proposition~\ref{pro:macmahon} follows from combining the second
equality in \eqref{equ:subgroups.abel} with \eqref{eq:mmpart}.
\end{rem}

Recall that for a map $T:X\rarr X$ we denote {by} $\OO_{T}(n) $
the number of closed orbits of $T$ of length~$n$. Let $\FF _T(n) = |
\{ x\in X \mid T^n(x) = x\}|$ denote the number of points of
period~$n$. Then
\begin{equation}\label{eq:ft}
  \FF_{T}(n) = \sum_{d|n}d \OO_{T}(d)
\end{equation}
and, by M\"obius inversion,
\begin{equation}\label{eq:mob}
\OO_T (n)=\frac{1}{n} \sum_{d|n}\mu \left(\frac{n}{d}\right)\FF _T (d).
\end{equation}
From \eqref{eq:ft}, 
\begin{equation}\label{eq:pt}
\pp_T(s):=\sum_{n=1}^{\infty}\FF_{T}(n) n^{-s}=\zeta(s)\dd_T(s-1). 
\end{equation}

Let now $r\in \N$ and $T _r$ be a map with orbit Dirichlet series
$\dd_{T_r}(s) = \zeta_{\Z^r}(s)$ as in \eqref{zetafree}.
\begin{cor}\label{lem:fta} For all $k\in\N_0$,
 \begin{equation*}
   \FF_{T_r} (p^k)= \sum_{j=0}^k p^{j} a_{p^j}(\Zp^r) =  \binom{r+k}{k}_{p}.
 \end{equation*}
\end{cor}

\begin{proof}
By \eqref{zetafree} and \eqref{eq:pt},
$\pp_{T_r}(s)=\zeta(s)\prod_{i=0}^{r-1}\zeta(s+1-i)=\zeta_{\Z ^{r+1}}(s).$
Together with Lemma~\ref{lem:subgroups}, this yields the result.
\end{proof}
Recall that $\lambda=(\lambda_1,\dots,\lambda_m)$ is a partition.  For
all $n\in\N$ we have $\FF_{ T_{\lambda_1}\times\cdots\times
   T_{\lambda_m}}(n)=\prod_{i=1}^m \FF _{T_{\lambda_i}}(n)$, so that,
by Corollary~\ref{lem:fta}, for a prime $p$ and $k\in\N_0$,
$$\FF_{T_{\lambda}} (p^k)=\prod_{i=1}^m
\binom{\lambda_i+k}{k}_{p}.$$ Using \eqref{eq:mob} we deduce, as
in the proof of \cite[Proposition~3.1]{PakapongpunWard/14}, that, for
$k>0$,
\begin{align*}
  \OO_{T_{\lambda}}(p^k) &= \frac{1}{p^k} \sum_{d| p^k}
  \mu\left(\frac{p^k}{d}\right)
  \FF_{T_{\lambda}}(d)\\
  &= \frac{1}{p^k} \left( \FF_{ T_{\lambda}}(p^k) -
    \FF_{T_{\lambda}}(p^{k-1})\right)\\
  &= \frac{1}{p^k} \left(\prod_{i=1}^m \binom{\lambda_i+k}{k}_{p} -
    \prod_{i=1}^m \binom{\lambda_i+k-1}{k-1}_{p}\right).
\end{align*}
Hence
\begin{align*}
  \dd_{T_{\lambda},p}(s)= \sum_{k=0}^\infty \OO_{
    T_{\lambda}}(p^k)t^k &= 1 + \sum_{k=1}^\infty \frac{1}{p^k}
  \left(\prod_{i=1}^m \binom{\lambda_i+k}{k}_{p} - \prod_{i=1}^m
    \binom{\lambda_i+k-1}{k-1}_{p}\right)t^k\\ &=
  \left(1-\frac{t}{p}\right) \sum_{k=0}^\infty \left(\prod_{i=1}^m
    \binom{\lambda_i+k}{k}_{p} \right)
  \left(\frac{t}{p}\right)^k.
\end{align*}
By substituting $(t/p,p)$ for $(x,q)$ in \eqref{eq:mmpart} and
setting $t=p^{-s}$, this may be rewritten as
$$\dd_{ T_{\lambda},p}(s)=
\frac{C_{\lambda}(p^{-1-s},p)}{\prod_{i=1}^{N}(1-p^{i-1-s})}.$$
The second statement in \eqref{equ:partition} follows from
\eqref{equ:C.general.part}. This concludes the proof of
Theorem~\ref{thm:partition}.  

\subsection{Proof of Theorem~\ref{cor:funeq}}
Given the expression \eqref{equ:partition} in
Theorem~\ref{thm:partition}, it is clear that a functional equation of
the form \eqref{equ:funeq.dir} holds if and only if $C_{\lambda}(x,q)$
satisfies a functional equation of the form~\eqref{equ:funeq.gen}. By
Proposition~\ref{pro:funeq.C.rectangle}, this holds if and only if
$\lambda$ is a rectangle. If $\lambda=(r^m)$, then, substituting
$(p^{-1-s},p)$ for $(x,q)$, this result implies that
\begin{equation}\label{eq:numfun}C_{r,m}(p^{1+s},p^{-1}) =
  p^{-r^2\binom{m}{2}+r(m-1)+sr(m-1)}
  C_{r,m}(p^{-1-s},p).\end{equation}
The functional equation \eqref{equ:fun} holds, as
$$\frac{1}{\prod_{i=1}^{rm}(1-p^{-i+1+s})} =
(-1)^{rm}p^{\binom{rm+1}{2}-rm-srm}\frac{1}{\prod_{i=1}^{rm}(1-p^{i-1-s})}$$
combined with \eqref{eq:numfun} gives
\begin{align*}
  \dd_{ T_r^{\times m},p}(s)\vert_{p \rarr p^{-1}} &= (-1)^{rm}
  p^{\binom{rm+1}{2}-rm-r^2\binom{m}{2}+r(m-1)-srm+sr(m-1)}
  \dd_{{ T}_r^{\times m},p}(s)\\&=(-1)^{rm}
  p^{m\binom{r+1}{2}-r-rs} \dd_{{T}_r^{\times m},p}(s).
  \qedhere \end{align*}

\section{Analytic properties and asymptotics}\label{sec:analytic}

In this section we exploit the combinatorial description of the
Dirichlet series $\dd_{ T _\lambda}(s)$ given in \eqref{equ:partition}
to deduce some of their key analytic properties.  Recall that $\lambda$ is a partition of $N = \sum_{i=1}^m \lambda_i$. In
the following, $f(n) \sim g(n)$ means that $\lim_{n\rarr
  \infty}f(n)/g(n) = 1$.
\begin{thm}\label{analytic}
\begin{enumerate}
\item The orbit Dirichlet series $\dd_{T_{\lambda}}(s)$ has
  abscissa of convergence $N$. If $m=1$ or $\lambda=(1,1)$, then
  it may be continued meromorphically to the whole complex plane;
  otherwise it has meromorphic continuation to $$\{s\in\C \mid
  \real(s)> N-2\}.$$ In any case, the continued function is
  holomorphic on the line $\{s\in\C \mid \real(s) = N\}$ except
  for a simple pole at $s=N$.
\item There exists a constant $K_{\lambda}\in\R_{>0}$ such that
$$\sum_{\nu\leq n}O_{ T_{\lambda}}(\nu) \sim K_{\lambda}
n^{N} \quad \textup{ as }n \rarr \infty.$$ 
\end{enumerate}
\end{thm}

\begin{proof}
  Recall that, for all $b\in\N_0$, the translate $\zeta(s-b)$ of
  Riemann's zeta function converges for $\real(s)>b+1$ and may be
  continued meromorphically to the whole complex plane. The continued
  function is holomorphic on the line $\{s\in\C \mid \real(s) = b+1
  \}$ except for a simple pole at $s=b+1$. This establishes all claims
  in (1) if $m=1$ or $\lambda=(1,1)$, as $\dd_{ T_r}(s) =
  \zeta_{\Z^r}(s)$ and $\dd_{ T_{(1,1)}}(s) =
  \frac{\zeta(s)^2\zeta(s-1)}{\zeta(2s)}$.

  Assume thus that $m \geq 2$ and $\lambda\neq (1,1)$ and recall the
  expression \eqref{equ:partition} for $\dd_{T_\lambda}(s)$.  The
  product $\prod_{p \textrm{ prime}}{C_{\lambda}(p^{-1-s},p)}$
  has abscissa of convergence $N-1$ and may be meromorphically
  continued to $\{s\in\C \mid \real(s) > N-2\}$.  Indeed, an
  Euler product of the form $$\prod_{p \textrm{ prime}} \left(1 +
    \sum_{(i,k)\in I}p^{i-ks}\right),$$ where $I \subset \N_0 \times
  \N$ is a finite (multi-)set, converges on $\left\{s\in\C \mid
    \real(s) > \alpha \right\}$, where $$\alpha =
  \max\left\{\frac{i+1}{k} \mid (i,k)\in I \right\},$$ and has a
  meromorphic continuation to $\left\{s\in\C \mid \real(s) > \beta
  \right\}$, where $$\beta = \max\left\{\frac{i}{k} \mid (i,k)\in
    I\right\};$$ see \cite[Lemmas~5.4 and 5.5]{duSWoodward/08}.  It
  follows from inspection of the Euler product
\begin{equation}\label{ep}
  \prod_{p \textrm{ prime}}{C_{\lambda}(p^{-1-s},p)} =\prod_{p \textrm{
      prime}}\left({\sum_{w\in S_{\lambda}}
      \prod_{j\in\Des(w)}p^{j-1-s}}\right)
\end{equation}
that the relevant maxima $ \alpha = {N-1}$ resp.\ $\beta =
N-2$ are both attained at the elements $w\in S_{\lambda}$
with $\Des(w)=\{N-1\}$. We note that
\begin{equation}\label{eq:nmax}
|\{w \in S_\lambda \mid \Des(w)=\{N-1\}\}|=m-1.
\end{equation}

As $$\prod_{p \textup{ prime}}
\frac{1}{\prod_{i=1}^{N}(1-p^{i-1-s})}= \prod_{i=1}^{N}
\zeta(s-i+1)$$ has abscissa of convergence $N > {\alpha}$,
this concludes the proof of~(1).

Statement (2) follows from~(1), for instance via the Tauberian theorem
\cite[Theorem~4.20]{duSG/00}.\qedhere
\end{proof}

If $m > 1$ and $\lambda\neq (1,1)$, then the meromorphic continuation
to $\beta = N-2$ is often{---}and conjecturally
always{---}best possible, as we now prove.
\begin{thm}\label{thm:nat.bound}
Assume that $\lambda\neq (1,1)$ and that either
\begin{itemize}
\item[(i)] $m>2$ or
\item[(ii)] $m=2$ and $\lambda_1=\lambda_2$ or
\item[(iii)]  $m=2$, $\lambda_1>\lambda_2$, and
  Conjecture~~\ref{con:nonzero} holds.
\end{itemize}
Then the orbit Dirichlet series $\dd_{T_{\lambda}}(s)$ has
  a natural boundary at $$\{s\in\C \mid \real(s)=N-2 \}.$$
\end{thm}
\begin{proof}
  We keep the notation as in Theorem~\ref{analytic} and set
  $${W}^\lambda (X,Y) = C_\lambda(X^{-1}Y,X) = \sum_{(i,k)\in
    { I}_{\lambda}}c_{i,k}X^iY^k \in \Z [X,Y]$$ for suitable ${
    I}_{\lambda} \subseteq \N_0^2$ and $c_{i,k}\in\N$. The Euler
  product \eqref{ep} then reads
\begin{equation}\label{ep.W}
 \prod_{p \textup{ prime}}{W}^{\lambda}(p,p^{-s}).
\end{equation}
To prove that under any of the assumptions (i)-(iii), the Euler
product \eqref{ep.W} has a natural boundary at $ \beta=N-2$ we
consider the \emph{ghost polynomial} associated to ${W}^{\lambda}$ and
prove that ${W}^{\lambda}$ is a polynomial of \emph{type
  I} (in case (i)) or \emph{type II} (in cases
(ii) and (iii)) in the terminology of
\cite[Section~5.2]{duSWoodward/08}.

We claim that the first factor of the ghost polynomial of
${W}^{\lambda}(X,Y)$ is, in notation close to the one used
in~\cite[Section~5.2]{duSWoodward/08},
\begin{equation}\label{equ:ghost}
 \widetilde{W}_1 ^{\lambda}(X,Y)=\sum_{(i,k)
  \in \mathfrak{l}_1 \cap{I}_{\lambda}} c_{i,k} X^i Y^k= 1+(m-1)X^{ \beta}Y.
\end{equation}
Here, $\mathfrak{l}_1$ is the line in $\R^2$ through $(0,0)$ and $(
\beta,1)$. It is characterized by the fact that its gradient $1/ 
\beta$ is minimal among the lines in $\R^2$ passing through $(0,0)$
and the points $(i,k)\in {I}_\lambda\setminus\{(0,0)\}$. Moreover,
$\mathfrak{l}_1 \cap {I}_{\lambda}=\{(0,0),(\beta,1)\}$ and
$c_{ \beta,1}=m-1$ (cf.\ \eqref{eq:nmax}), which
proves~\eqref{equ:ghost}. Setting $U = X^{\beta}Y$, we
obtain $$\widetilde{{W}}_1^{\lambda}(U) = 1 + (m-1)U\in\Z[U].$$

If $m > 2$, then $\widetilde {W}_1^{\lambda}(U)$ is not cyclotomic,
whence ${W}^{\lambda}(X,Y)$ is a polynomial of type I in the
parlance of~\cite[p.~127]{duSWoodward/08}. Without loss of generality
we may divide ${W}^{\lambda}$ by any unitary factors it may have;
cf.\ Conjecture~\ref{con:master}. Indeed, if ${W}^\lambda = f
{V}^\lambda$ for $f\in\Z[X,Y]$ unitary, then the Newton polygon of
${W}^\lambda$ is the Minkowski sum of the Newton polygons of $f$ and
${V}^\lambda$. The former, however, is a segment of a line
in~$\R^2$. As $\wt{W}_1^\lambda$ does not have a unitary factor, the
slope of this line is strictly larger than $1/{\beta}$, whence
$\wt{W}_1^\lambda = \wt{ V}_1^\lambda$, i.e. the first factors of
the ghosts of ${W^\lambda}$ and ${V^\lambda}$ coincide.

Assuming thus, as we may, that ${W}^\lambda$ has no unitary
factors, \cite[Theorem~5.6]{duSWoodward/08} yields that $ \beta $ is
a natural boundary for \eqref{ep.W} and thus for $\dd_{
  T_{\lambda}}(s)$. This concludes the proof in case~(i).

Turning to cases (ii) and (iii) we now assume that $m=2$. Hence
$\widetilde {W}_1^{\lambda}(U) =1+U$ \emph{is} cyclotomic. In
particular, ${W}^{\lambda}(X,Y)$ is not of type I.  We claim that it
is a polynomial of type II in the sense
of~\cite[p.~127]{duSWoodward/08}. To prove this, we check that the
hypotheses of \cite[Corollary~5.15]{duSWoodward/08} are satisfied. To
this end, we verify that ${W}^{\lambda}(X,Y)$ is such that Hypotheses
1 and 2 defined on \cite[p.~134]{duSWoodward/08} are satisfiable. The
polynomial $A(U)=1+\sum_{\frac{n_k}{k}={\beta}} c_{{n_k},k} U^k=1+U$
(cf.\ \cite[p.~134]{duSWoodward/08}) obviously has a unique root
$\omega=-1$. It is simple, so in particular satisfies Hypothesis~1.

Hypothesis~2 is equivalent to $\real
\left(-\frac{B_\gamma (\omega)}{\omega A' (\omega)}\right)<0$, hence to
$B_\gamma (\omega)<0$, where $$\gamma:= \min\{n\in \N_0 \mid
B_n(\omega)\neq 0\}$$ and, for $n \in \N_0$ and $(n_j,j)\in {
  I}_{\lambda}$ such that $n_j/j={\beta}$ and $j$ is minimal with this
property,
\begin{equation*}
B_n(U)=\sum_{n_j k - ij=n} c_{i,k} U^k=\sum_{\beta k - i=n} c_{i,k}
U^k;
\end{equation*} 
cf.\ \cite[(5.12)]{duSWoodward/08}. Note that
$B_0(U)=A(U)=1+U$.

Recall that
$${W}^{\lambda}(X,Y) = \sum_{w\in S_{\lambda}} X^{\maj(w) -
  \des(w)} Y^{\des(w)}= \sum_{(i,k)\in {I}_\lambda}c_{i,k} X^i
Y^k.$$ For $(i,k)\in {I}_{\lambda}$, we thus have $c_{i,k}=|\{w
\in S_{\lambda}\mid k=\des(w),\,i=\maj(w)-k\}|$. 

We claim that $\gamma=1$.  If $(i,k)$ satisfies
\begin{equation}\label{cdt2}
\beta k-i=1,
\end{equation}
then clearly $k\neq 0$. If $k=2$, then \eqref{cdt2} necessitates
$i=2N-5$, that is $\maj(w)=2N-3$. But there is no element $w \in
S_{\lambda}$ such that $\des(w)=2$ and $\maj(w)=2N-3$. Indeed, such an
element would need to have descents at the consecutive positions $N-2$
and $N-1$ (as $\maj(w)=2N-3=(N-1)+(N-2)$), which is clearly impossible
for a word in
$\{\underbrace{\bfone,\dots,\bfone}_{\lambda_1},\underbrace{\bftwo,\dots,\bftwo}_{\lambda_2}\}$. A
similar argument excludes pairs $(i,k)$ that satisfy \eqref{cdt2} and
for which~$k>2$.  To determine $B_1(U)$ we thus need to
determine $$c_{ \beta-1,1}=|\{w \in S_{\lambda} \mid
\des(w)=1,\,\maj(w)=N-2\}|,$$ i.e.\ to enumerate the multiset
permutations with no descent in the first $N-3$ positions and ending
in $\ldots \bftwo\bfone\bftwo$ or $\ldots \bftwo\bfone\bfone$. If
$\lambda_2>1$, then there are exactly two such words; if
$\lambda_2=1$, then only the second option occurs. So $B_{1}(U)=2U$
resp.\ $B_{1}(U)=U$. In any case, $\gamma=1$ and
$B_\gamma(\omega)=-2<0$ resp.\ $B_\gamma(\omega)=-1<0$. Hence
Hypothesis~2 is satisfied.

Since Hypotheses 1 and 2 are satisfied and $1=\gamma \geq j=1$,
\cite[Corollary~5.15]{duSWoodward/08} implies that ${
  W}^{\lambda}(X,Y)$ is of type II. Thus in case (ii) for
$\lambda_1=\lambda_2$ odd, and in case (iii), as ${W}^{\lambda}(X,Y)$
has no unitary factors, \cite[Theorem~5.6]{duSWoodward/08} yields that
$\beta$ is a natural boundary for \eqref{ep.W} and thus
for~$\dd_{T_{\lambda}}(s)$.  In case (ii) for $\lambda_1=\lambda_2$
even, Proposition~\ref{pro:shape} asserts that a unique unitary factor
exists: ${W}^{\lambda}(X,Y)=(1+X^{\lambda_1 -1}Y){{W}'} ^\lambda(X,Y)$
for some ${W'}^\lambda \in \Z [X,Y]$. But $ \beta=N-2>\lambda_1 -1$,
so the minimal gradient for ${W'}^\lambda(X,Y)$ is still $\beta$. Thus
also in this case \cite[Theorem~5.6]{duSWoodward/08} implies that
$\beta$ is a natural boundary.
\end{proof}

\section{Connection with Igusa functions and the special case
  $\lambda=(1^m)$}\label{sec:one}
Taking $\lambda=(1^m)$ corresponds to considering the $m$-th power of
a map $T=T_1$ whose orbit Dirichlet series $\dd_T(s)$ is the Riemann
zeta function $\zeta(s)$. In this case, Theorem~\ref{thm:partition}
reads
\begin{equation*}
\dd_{ T^{\times m}}(s) = \prod_{p \textup{ prime}}
\frac{C_m(p^{-1-s},p)}{\prod_{i=1}^m(1-p^{i-1-s})}=\prod_{p \textup{ prime}}
\frac{\sum_{w\in S_m} {p^{(-1-s)\des(w)+\maj(w)}}}{\prod_{i=1}^m(1-p^{i-1-s})}.
\end{equation*}
where $C_m(x,q)=C_{1,m}(x,q)$ is Carlitz's $q$-Eulerian polynomial.

More generally one may, for $a\in\R_{\geq 0}$, consider a map $_a T$
such that $\dd _ {_a T}(s)=\zeta(s-a)$.  Then $$\OO_{_a T}(p^k)=p^{ak}
\quad \text{ and }\quad \FF_{_a T} (p^k)=\sum_{j=0}^k p^j
p^{aj}=\binom{1+k}{1}_{p^{a+1}}.$$

The orbit Dirichlet series of the $m$-th power of $_a T$ is thus MacMahon's generating series~\eqref{eq:mmpart} for $\lambda=(1^m)$ and $(x,q)=(p^{-1-s},p^{a+1})$:
\begin{equation}\label{eq:Ta}
\dd_{ _a T^{\times m}}(s) = \prod_{p \textup{ prime}}
\frac{C_m(p^{-1-s},p^{a+1})}{\prod_{i=1}^m(1-p^{(a+1)i-1-s})}=\prod_{p \textup{ prime}}
\frac{\sum_{w\in S_m} {p^{(-1-s)\des(w)+(a+1)\maj(w)}}}{\prod_{i=1}^m(1-p^{(a+1)i-1-s})}.
\end{equation}

A formula for the local factors $\dd_{_a T ^{\times m},p}(s)$ of
$\dd_{_a T ^{\times m}}(s)$ appears in
\cite[p.~41]{PakapongpunWard/14}, where they are called $E_p (s)$, and
suffers from a transcript error in the definition of the
expression~$A_b$. Moreover, no combinatorial interpretation is
given. \cite[Theorem~4.1]{PakapongpunWard/14} is
Theorem~\ref{thm:nat.bound} in the special case~$\lambda=(1,1,1)$.

Each factor of the Euler product~\eqref{eq:Ta} is an instance of an Igusa
function:
\begin{align*}
\dd_{_a T^{\times m},p}(s)&=
  \frac{C_m(p^{-1-s},p^{a+1})}{\prod_{i=1}^m(1-p^{(a+1)i-1-s})} =\frac{\sum_{w\in
      S_m} \prod_{j\in\Des(w)}p^{(a+1)j-1-s}}{\prod_{i=1}^m(1-p^{(a+1)i-1-s})}\nonumber\\
  &=\frac{1}{1-p^{(a+1)m-1-s}}\sum_{I \subseteq [m-1]}\binom{m}{I}\prod_{i\in
    I}\frac{p^{(a+1)i-1-s}}{1-p^{(a+1)i-1-s}}\in\Q(p,p^{-s}).
\end{align*}
In the terminology of \cite[Definition~2.5]{SV1/15} it would be called
$I_m(1;(p^{(a+1)i-1-s})_{i=1}^m)$ and so \eqref{equ:fun} in
Theorem~\ref{cor:funeq} follows in this case
from~\cite[Proposition~4.2]{SV1/15}.

In the case of a general partition, we are not aware of a simple
expression of the local factors of the orbit Dirichlet series of the
product of such ``shifted maps''. Turning back to the case $a=0$ and
general partition $\lambda$, the local factors of
\eqref{equ:partition} may be rewritten as
\begin{equation*}\label{equ:C=I}
  \dd_{T_{\lambda},p}(s)
  =\frac{1}{1-p^{N-1-s}}\sum_{I \subseteq [N-1]}\nu_{\lambda,I}
  \prod_{i\in I}\frac{p^{i-1-s}}{1-p^{i-1-s}}\in\Q(p,p^{-s})
\end{equation*}
where $\nu_{\lambda,I}=|\{w \in S_{\lambda}\mid \Des(w)\subseteq
I\}|$.  We are not aware of a simple expression, say in terms of
multinomial coefficients, for $\nu_{\lambda,I}$ if~$\lambda$ is not of
the form~$(1^m)$.

\section{Reduced orbit Dirichlet series: setting $p=1$}\label{sec:red}
Viewing the Euler factors of \eqref{equ:partition} as bivariate
rational functions in $p$ and $t=p^{-s}$, one may evaluate them at
$p=1$ whilst leaving $t$ as an independent variable. Motivated by the
notion of \emph{reduced zeta functions of Lie algebras} introduced
in~\cite{Evseev/09} we thus define the \emph{reduced orbit Dirichlet
  series} $$\dd_{T_{\lambda},\red}(t) :=
\frac{C_{\lambda}(t,1)}{(1-t)^{N}}\in\Q(t).$$ It seems remarkable that
for $\lambda=(1^m)$ the reduced orbit Dirichlet series
$\dd_{T_1^{\times m},\red}(t)$ is the Hilbert series of the
Stanley-Reisner ring of a simplicial complex. Indeed, let $k$ be any
field, write $\sd(\Delta_{m-1})$ for the barycentric subdivision of
the $(m-1)$-simplex $\Delta_{m-1}${---}or, equivalently, the
Coxeter complex of type $A_{m-1}${---}, with Stanley-Reisner (or face) ring
$k[\sd(\Delta_{m-1})]$; see, for instance, \cite[Ch.~III,
  Sec.~4]{Stanley/96}. The fact that the $m$-th Eulerian polynomial
(cf.\ \eqref{equ:eulerian}) is the generating function of the
$h$-vector of $\sd(\Delta_{m-1})$ is reflected in the following fact.

\begin{pro}
$$\dd_{T_1^{\times m},\red}(t) = \frac{A_m(t)/t}{(1-t)^m} =
  \Hilb(k[\sd(\Delta_{m-1})],t).$$
\end{pro}

Similarly, we observe that for $\lambda=(r,r)$, the polynomial
$C_{r,2}(t,1)$ may be viewed as the $h$-vector of the $r$-dimensional
type-B simplicial associahedron $Q_n^{\textup{B}}$; cf.\
\cite[Corollary~1]{Simion/03}.

\begin{acknowledgements}
  We are grateful to Tom Ward, who pointed us to
  \cite{PakapongpunWard/14} and suggested to generalize the set-up
  there, and to Michael Baake, who introduced us to Ward. Both Baake
  and Ward made valuable comments about earlier drafts of this
  paper. We thank Laurent Habsieger, Christian Krattenthaler, Tobias
  Rossmann, and Dennis Stanton for helpful remarks. We were supported
  by the German-Israeli Foundation for Scientific Research and
  Development, through grant no.~1246.

  This is a pre-print of an article published in \emph{Monatshefte
    f\"ur Mathematik}. The final authenticated version is available
  online at: \url{https://doi.org/10.1007/s00605-017-1128-9}
\end{acknowledgements}

\def\cprime{$'$}
\providecommand{\bysame}{\leavevmode\hbox to3em{\hrulefill}\thinspace}
\providecommand{\MR}{\relax\ifhmode\unskip\space\fi MR }
\providecommand{\MRhref}[2]{%
  \href{http://www.ams.org/mathscinet-getitem?mr=#1}{#2}
}
\providecommand{\href}[2]{#2}

\end{document}